\documentclass[12pt]{amsart}

\usepackage{amsmath,amssymb,amsbsy,amsfonts,amsthm,latexsym,
                        amsopn,amstext,amsxtra,euscript,amscd,mathrsfs,color,bm,cite}

\usepackage{float}
\usepackage[english]{babel}
\usepackage{mathtools}
\usepackage{todonotes}
\usepackage{url}
\usepackage[colorlinks,linkcolor=blue,anchorcolor=blue,citecolor=blue,backref=page]{hyperref}

\usepackage[norefs,nocites]{refcheck}

\newtheorem{theorem}{Theorem}
\newtheorem{lemma}[theorem]{Lemma}

\numberwithin{equation}{section}
\numberwithin{theorem}{section}
\numberwithin{table}{section}

\numberwithin{figure}{section}

\newfont{\teneufm}{eufm10}
\newfont{\seveneufm}{eufm7}
\newfont{\fiveeufm}{eufm5}
\newfam\eufmfam
                \textfont\eufmfam=\teneufm \scriptfont\eufmfam=\seveneufm
                \scriptscriptfont\eufmfam=\fiveeufm
\def\bbbc{{\mathchoice {\setbox0=\hbox{$\displaystyle\rm C$}\hbox{\hbox
to0pt{\kern0.4\wd0\vrule height0.9\ht0\hss}\box0}}
{\setbox0=\hbox{$\textstyle\rm C$}\hbox{\hbox
to0pt{\kern0.4\wd0\vrule height0.9\ht0\hss}\box0}}
{\setbox0=\hbox{$\scriptstyle\rm C$}\hbox{\hbox
to0pt{\kern0.4\wd0\vrule height0.9\ht0\hss}\box0}}
{\setbox0=\hbox{$\scriptscriptstyle\rm C$}\hbox{\hbox
to0pt{\kern0.4\wd0\vrule height0.9\ht0\hss}\box0}}}}
\def\bbbq{{\mathchoice {\setbox0=\hbox{$\displaystyle\rm
Q$}\hbox{\raise 0.15\ht0\hbox to0pt{\kern0.4\wd0\vrule
height0.8\ht0\hss}\box0}} {\setbox0=\hbox{$\textstyle\rm
Q$}\hbox{\raise 0.15\ht0\hbox to0pt{\kern0.4\wd0\vrule
height0.8\ht0\hss}\box0}} {\setbox0=\hbox{$\scriptstyle\rm
Q$}\hbox{\raise 0.15\ht0\hbox to0pt{\kern0.4\wd0\vrule
height0.7\ht0\hss}\box0}} {\setbox0=\hbox{$\scriptscriptstyle\rm
Q$}\hbox{\raise 0.15\ht0\hbox to0pt{\kern0.4\wd0\vrule
height0.7\ht0\hss}\box0}}}}
\def\bbbt{{\mathchoice {\setbox0=\hbox{$\displaystyle\rm
T$}\hbox{\hbox to0pt{\kern0.3\wd0\vrule height0.9\ht0\hss}\box0}}
{\setbox0=\hbox{$\textstyle\rm T$}\hbox{\hbox
to0pt{\kern0.3\wd0\vrule height0.9\ht0\hss}\box0}}
{\setbox0=\hbox{$\scriptstyle\rm T$}\hbox{\hbox
to0pt{\kern0.3\wd0\vrule height0.9\ht0\hss}\box0}}
{\setbox0=\hbox{$\scriptscriptstyle\rm T$}\hbox{\hbox
to0pt{\kern0.3\wd0\vrule height0.9\ht0\hss}\box0}}}}
\def\bbbs{{\mathchoice
{\setbox0=\hbox{$\displaystyle     \rm S$}\hbox{\raise0.5\ht0\hbox
to0pt{\kern0.35\wd0\vrule height0.45\ht0\hss}\hbox
to0pt{\kern0.55\wd0\vrule height0.5\ht0\hss}\box0}}
{\setbox0=\hbox{$\textstyle        \rm S$}\hbox{\raise0.5\ht0\hbox
to0pt{\kern0.35\wd0\vrule height0.45\ht0\hss}\hbox
to0pt{\kern0.55\wd0\vrule height0.5\ht0\hss}\box0}}
{\setbox0=\hbox{$\scriptstyle      \rm S$}\hbox{\raise0.5\ht0\hbox
to0pt{\kern0.35\wd0\vrule height0.45\ht0\hss}\raise0.05\ht0\hbox
to0pt{\kern0.5\wd0\vrule height0.45\ht0\hss}\box0}}
{\setbox0=\hbox{$\scriptscriptstyle\rm S$}\hbox{\raise0.5\ht0\hbox
to0pt{\kern0.4\wd0\vrule height0.45\ht0\hss}\raise0.05\ht0\hbox
to0pt{\kern0.55\wd0\vrule height0.45\ht0\hss}\box0}}}}
\def\bbbz{{\mathchoice {\hbox{$\sf\textstyle Z\kern-0.4em Z$}}
{\hbox{$\sf\textstyle Z\kern-0.4em Z$}} {\hbox{$\sf\scriptstyle
Z\kern-0.3em Z$}} {\hbox{$\sf\scriptscriptstyle Z\kern-0.2em
Z$}}}}

\def\squareforqed{\hbox{\rlap{$\sqcap$}$\sqcup$}}
\def\qed{\ifmmode\squareforqed\else{\unskip\nobreak\hfil
\penalty50\hskip1em\null\nobreak\hfil\squareforqed
\parfillskip=0pt\finalhyphendemerits=0\endgraf}\fi}

\def\cH{{\mathcal H}}

\def\cJ{{\mathcal J}}
\def\cK{{\mathcal K}}

\def\cM{{\mathcal M}}
\def\cN{{\mathcal N}}

\def\cU{{\mathcal U}}

\def\cX{{\mathcal X}}

\def\le{\leqslant}

\def\ge{\geqslant}

\newcommand{\ignore}[1]{}

\def\e{\mathbf{e}}

\def \F{\mathbb{F}}

\def \Z{\mathbb{Z}}

\def \Z{\mathbb{Z}}

\def \balpha{\bm{\alpha}}
\def \bbeta{\bm{\beta}}

\def\mand{\qquad\mbox{and}\qquad}

\def\\{\cr}
\def\({\left(}
\def\){\right)}
\def\fl#1{\left\lfloor#1\right\rfloor}

\def\e{\mathbf{e}}

\def\eps{\varepsilon}

\begin{document}

\title[Congruences and exponential sums]
{On some congruences and exponential sums}

\author{Moubariz~Z.~Garaev}
\address{Centro  de Ciencias Matem{\'a}ticas,  Universidad Nacional Aut\'onoma de
M{\'e}\-xico, C.P. 58089, Morelia, Michoac{\'a}n, M{\'e}xico}
\email{garaev@matmor.unam.mx}

\author[I. E. Shparlinski] {Igor E. Shparlinski}
\address{School of Mathematics and Statistics, University of New South Wales, Sydney, NSW 2052, Australia}
\email{igor.shparlinski@unsw.edu.au}

\begin{abstract}  
Let $\varepsilon>0$ be a fixed small constant, $\F_p$ be the finite  field of $p$ elements for  prime  $p$.  
We consider additive and multiplicative problems in $\F_p$ that  
involve intervals and arbitrary sets. Representative examples of our results are as follows. 
Let $\cM$ be an arbitrary subset of $\F_p$.
If $\#\cM >p^{1/3+\varepsilon}$ and $H\ge p^{2/3}$  or if  $\#\cM >p^{3/5+\varepsilon}$ and 
$H\ge p^{3/5+\varepsilon}$ then all, but $O(p^{1-\delta})$ elements
of $\F_p$ can be represented in the form $hm$ with $h\in [1, H]$ and $m\in \cM$, 
where $\delta> 0$ depends only on $\varepsilon$. 
Furthermore, let $\cX$ be an arbitrary interval of length $H$ and $s$ be a  fixed positive integer. If 
$$
H>  p^{17/35+\varepsilon}, \quad \#\cM > p^{17/35+\varepsilon}.
$$  
then the number $T_6(\lambda)$ of solutions of the congruence
\begin{align*}
\frac{m_1}{x_1^s}+ \frac{m_2}{x_2^s}+ \frac{m_3}{x_3^s}&+\frac{m_4}{x_4^s}+ \frac{m_5}{x_5^s}+\frac{m_6}{x_6^s}
\equiv \lambda\mod  p, \\
  m_i\in \cM, & \ x_i  \in \cX, \quad i =1, \ldots, 6,
\end{align*}
satisfies
$$
T_6(\lambda)=\frac{H^6(\#\cM)^6}{p}\(1+O(p^{-\delta})\),
$$
where  $\delta> 0$ depends only on $s$ and $\varepsilon$. 
\end{abstract}


\maketitle

\tableofcontents
\section{Introduction}

\subsection{Motivation}

Here we consider some question related to congruences modulo 
 a sufficiently large prime number $p$, which involve products and inverses of 
 elements from short intervals or arbitrary sets. In particular as one of the applications 
 we improve a result on cardinlaities of product of a set and an interval from~\cite{GarKar2}  and a bound of exponential sums from~\cite{BagSh}. 

In what follows, $p$ is a large prime number,  $\F_p$ is the field of residue classes modulo $p$.  By $\F_p^*$ we denote
the set of non-zero elements of $\F_p$.  
Let 
\begin{equation}
\label{eq: init int H}
\cH = \{1, \ldots, H\} ,
\end{equation}
 be an initial interval of $H\ge 1$ consecutive integers and
let 
\begin{equation}
\label{eq: set M}
\cM \subseteq \F_p^* \mand  \# \cM = M
\end{equation}
 be a subset of $\F_p^*$ of cardinality $M$.

We are interested in a natural question of investigating the size of set of ratios 
\begin{equation}
\label{eq: ratio set}
\cM/\cH =  \{m/h:~h \in \cH, \ m \in \cM\}.
\end{equation}
considered as subset of $\F_p$ and in particular investigating the conditions of the sizes
$H$ and $M$ which guarantee that all but $o(p)$ elements of $\F_p$ are represented.

In fact, since $\cM$ is an arbitrary set, we see that this  
is equivalent to the question about the size  values set of the product set
\begin{equation}
\label{eq:  prod  set}
\cH \cM = \{hm:~h \in \cH, \ m \in \cM\}
\end{equation}
which is slightly more convenient to study. 

This question   of estimating the size of the ratio set~\eqref{eq: ratio set} or, alternatively of  
the product set~\eqref{eq: prod set}, 
had first been addressed by Garaev and Karatsuba~\cite{GarKar}. In particular,
a special case of~\cite[Theorem~2]{GarKar} implies that if
$H= M > p^{2/3-1/192+ \varepsilon}$ for some fixed $\varepsilon >0$, then
$\#\(\cH \cM\) = p + o(p)$.

Furthermore, if $\cM= \cH$ then by~\cite[Theorem~2]{GarKar}
the same holds for $H= M > p^{5/8+ \varepsilon}$. This condition has been relaxed in~\cite{GarKar2} to the optimal one $H= M > p^{1/2+ \varepsilon}$, 
see also~\cite{Shp1} for many other results on this and related topics.

This problem  has also attracted attention of Bourgain. It has been mentioned in~\cite[Page~480]{CillGar}
that in the case of arbitrary sets $\cM$ Bourgain (unpublished) improved the exponent $2/3-1/192$ to $5/8$.

Furthermore, we also consider a Waring-type problem with elements of the 
ration set~\eqref{eq: ratio set}. Namely for an integer $s \ge 1$ we study representations 
of all elements of $\F_p$ by sums of a small number of powers $(m/h)^s$, 
with $h \in \cH$, $m \in \cM$. In fact, we study this question in a broader generality 
for fractions of the form $m/x^s$ where the numerator is from an arbitrary set (and is
not necessary a perfect $s$-th power) and more importantly where $x$ runs though an 
arbitrary interval, not necessary at the origin. 

For this, we first  consider exponential sums with products $m/x^{s}$, with some fixed integer $s >0$, where $m$ is from an arbitrary set 
 $\cM \subseteq \F_p$ and $x$ is from an arbitrary interval  
 \begin{equation}
\label{eq: int X}
\cX = L + \cH = \{L+1, \ldots, L+H\}\subseteq  \{1, \ldots, p-1\}
\end{equation}
of $H$ consecutive integers, 
 which this time is not necessary at the origin (and such that $0 \notin \cX$). 

\subsection{New results}
Here we get the following improvement of the aforementioned results on the size of the product set $\cH\cM$. 

\begin{theorem}
\label{thm: HM}  Let  $\varepsilon >0$ be fixed.
Assume that  cardinalities $H$ and $M$ of an interval $\cH$ as in~\eqref{eq: init int H} and   
a set $\cM$ as in~\eqref{eq: set M} satisfy
$$
H\ge p^{2/3},  \qquad HM \ge p^{1+\varepsilon} 
$$
or
$$
H<p^{2/3},  \qquad M \ge p^{1/3}, \qquad  HM^{1/4} \ge p^{3/4+\varepsilon} .
$$
Then there is some $\eta>0$ which depends only on $\varepsilon$ such that   
$$
\#\(\cH \cM\) = p + O\( p^{1-\eta}\).
$$
\end{theorem}

It is easy to see from the proof of Theorem~\ref{thm: HM} that one can take $\eta = c \varepsilon$ 
for some absolute constant $c > 0$.

Next we consider exponential sums with  ratios  $m/x^s$ with $x$ running through a 
shifted interval $\cX$ as in~\eqref{eq: int X}. 

A bound on such sums has been given  in~\cite[Theorem~3.1]{BagSh} which we now  improve.  

It is convenient to define 
$$
\e_p(z) = \exp(2 \pi i z/p).
$$

\begin{theorem}
\label{thm: KloostFrac}  Let  $s$ be a fixed positive integer. 
Then for an interval $\cH$  as in~\eqref{eq: init int H}  a set $\cM$ as in~\eqref{eq: set M},  for any fixed positive integer $\ell$ the following bound holds
$$
\sum_{m\in\cM}\left|\sum_{  x\in L+\cH}\e_p(amx^{-s})\right|\le HM\(\frac{p}{MH^{2\ell/(\ell+1)}}+\frac{1}{M}\)^{1/(2\ell)} p^{o(1)}.
$$
\end{theorem}

We now combine some ideas and tools used in the proof of Theorem~\ref{thm: HM} 
to derive the following asymptotic formula.

\begin{theorem}
\label{thm: hexagon}  
Let $\ell$ and $s$ be  fixed positive integers and let $\varepsilon >0$ be a fixed real number. 
Assume that
$$
H^{24/17}M^{11/17}>p^{1+\varepsilon}, \quad H^{9/5} M^{2/5}>p^{1+\varepsilon}, \quad H^{6/5}M>p^{1+\varepsilon}.
$$
Then  for any sets $\cM_i \subseteq \F_p^*$, $i =1, \ldots 6$, with $\#\cM_i=M$, integers $L_i$,  interval $\cH$ 
as in~\eqref{eq: init int H} and  integer $\lambda$ the number $T_6(\lambda)$ of solutions of the congruence
\begin{align*}
\frac{m_1}{x_1^s}+ \frac{m_2}{x_2^s}+\frac{m_3}{x_3^s}&+\frac{m_4}{x_4^s}+ \frac{m_5}{x_5^s}+\frac{m_6}{x_6^s}\equiv \lambda\mod  p, \\
  m_i\in \cM_i, & \ x_i  \in L_i+ \cH, \quad i =1, \ldots, 6,
\end{align*}
satisfies
\begin{equation}
\label{eqn: T6 asympt}
T_6(\lambda)=\frac{H^6M^6}{p}\Bigl(1+O(p^{-\delta})\Bigr)
\end{equation}
where  $\delta> 0$ depends only on $s$ and $\varepsilon$. 
\end{theorem}

In particular, the asymptotic formula~\eqref{eqn: T6 asympt} holds for 
$$
H= M > p^{17/35+\varepsilon}.
$$  

We use this opportunity to pose an open question about obtaining an 
asymptotic formula for the $5$-term analogues of $T_6(\lambda)$ for 
 $H= M <p^{1/2}$.  This is interesting and still open even in the special case 
 $\cM_1 = \ldots = \cM_5$ and $L_1 = \ldots =L_5 = 0$.

\subsection{Notation and conventions}

We recall
that the notations $U = O(V)$,  $U \ll V$ and  $V \gg U$  are
all equivalent to the statement that $|U| \le c V$ holds
with some constant $c> 0$.

Any implied constants in symbols $O$, $\ll$
and $\gg$ may occasionally, where obvious, depend on the
integer parameters $\ell$, $r$ and $s$  and  the real parameter $\varepsilon >0$, 
 and are absolute otherwise.

Finally $U= o(V)$ means that $U \le \psi(V)V$ for some
function $\psi$ such that $\psi(V) \to 0$ as $V \to \infty$.

We use    $\tau(r)$ to denote the  the number of integer positive divisors of an
integer $r \ne 0$,    and recall the classical bound
\begin{equation}
\label{eq: tau}
\tau(r) = |r|^{o(1)},
\end{equation}
see, for example~\cite[Equation~(1.81)]{IwKow}).

Throughout the paper, we always assume that $\F_p$ is represented
by the set $\{0, 1, \ldots, p-1\}$ and we freely alternate between equations
in $\F_p$ and congruences modulo $p$.

\section{Preliminaries}

\subsection{Multiplicative congruences with initial intervals}

We need the following bound from the work of  Banks and Shparlinski~\cite[Theorem~2.1]{BaSh}, which applies to coincidences 
in products of elements from an initial interval  and an arbitrary subset $\cM$ of $\F_p^*$. 

Let  $J(\cH,\cM)$ be the number of solutions to the congruence
$$
h_1m_1\equiv h_2m_2\mod  p,\quad h_1,h_2\in \cH,\ m_1,m_2\in\cM,
$$
with  $\cH$ as in~\eqref{eq: init int H} and   $\cM$ as in~\eqref{eq: set M}. 

\begin{lemma}
\label{lem: BS-Energy} The following bound holds
\begin{align*}
J(\cH&,\cM)\\
&\ll \left\{\begin{array}{llll}
H^2M^2/p^{-1}+HMp^{o(1)}, &  \text{if }  H\ge p^{2/3},\\
H^2M^2p^{-1} +HM^{7/4}p^{-1/4+o(1)}+M^2, &  \text{if }   H<p^{2/3}\\
&  \quad   \text{ and } M\ge p^{1/3},\\
HMp^{o(1)}+M^{2}, &  \text{if }   H<p^{2/3}\\
&  \quad  \text{ and } M<p^{1/3}.
\end{array}
\right.
\end{align*}
\end{lemma}

From Lemma~\ref{lem: BS-Energy} we easily derive the following consequence:

\begin{lemma}
\label{lem: IntSet-Energy}
 Assume that  cardinalities $K$ and $M$ of an initial interval $\cK = \{1, \ldots, K\} $ and a set $\cM \subseteq  \F_p^*$ satisfy
$$
p > K\ge p^{2/3},  \qquad  KM \ge p,
$$
or
$$
K<p^{2/3},  \qquad M \ge p^{1/3}, \qquad  KM^{1/4} \ge p^{3/4}.
$$
Then we have
$$
J(\cK,\cM)
\le K^2M^2 p^{-1+o(1)}.
$$
\end{lemma}

We now consider triple products $j k m $ where as before,  $j \in \cJ$,  $k \in \cK$, $m \in \cM$,
where $\cJ$ and $\cK$ are initial intervals of $J$ and $K$ consecutive integers of the type similar to~\eqref{eq: init int H} and $\cM$
is an arbitrary set as in~\eqref{eq: set M}.
In particular, let
\begin{align*}
R(\cJ, \cK, \cM) =\# \{j_1 k_1 m_1 & \equiv j_2 h_2 m_2 \mod  p:\\
& ~j_1, j_2 \in \cJ, \, k_1,k_2 \in \cK, \, m_1,m_2 \in \cM\} .
 \end{align*}

\begin{lemma}
\label{lem:TripleProd}  Let $r\ge 1$ be a fixed integer.
 Assume that an initial interval $\cJ$  is of cardinality $J = \fl{p^{1/r}}$
 and   that  cardinalities $K$ and $M$ of an initial  interval $\cK$ and a set $\cM \subseteq \F_p^*$
 satisfy
$$
K\ge p^{2/3},  \qquad KM \ge p, 
$$
or
$$
K<p^{2/3},  \qquad M \ge p^{1/3}, \qquad  KM^{1/4} \ge p^{3/4}.
$$
Then we have
$$
R(\cJ, \cK, \cM) = \frac{J^2K^2 M^2} {p-1} \( 1+ O\( K^{-1/r}  p^{3/(8r)+ o(1)}\)\).
$$
\end{lemma}

\begin{proof} For a set $\cU  \subseteq \F_p^*$ we define the following
multiplicative character sum
$$
S_\cU(\chi) = \sum_{u \in \cU} \chi(u),
$$
see~\cite[Chapter~3]{IwKow} for a background on characters. 
Then, using the orthogonality of characters, we easily express $R(\cJ, \cK, \cM)$ as
\begin{equation}
\label{eq:R Delta}
R(\cJ, \cK, \cM) = \frac{J^2K^2 M^2} {p-1} + \frac{1}{p} \Delta, 
\end{equation}
where
$$
\Delta = \sum_{\chi \ne\chi_0} \left| S_\cJ(\chi)\right|^2
 \left| S_\cK(\chi)\right|^2 \left| S_\cM(\chi)\right|^2,
$$
with $\chi$ running over all multiplicative characters of  $ \F_p^*$ except
for the principal character $\chi_0$.

We now fix  some real parameter $\gamma> 0$, to be chosen later.

By the special case of the Burgess bound, see~\cite[Theorem~12.6]{IwKow},
we have
$$
 \left| S_\cK(\chi)\right| \le K^{1/2} p^{3/16+ o(1)}.
$$
We also have a trivial bound
$$
 \left| S_\cM(\chi)\right| \le M.
$$  
Hence, for $\gamma < 2$ we have
$$
\Delta \le  K^{\gamma/2} M^{\gamma} p^{3\gamma/16+ o(1)} \sum_{\chi \ne\chi_0} \left| S_\cJ(\chi)\right|^2
 \left| S_\cK(\chi)\right|^{2-\gamma} \left| S_\cM(\chi)\right|^{2-\gamma}.
$$

Next,  by H{\"o}lder's inequality (and also dropping the now
unnecessary condition $\chi \ne\chi_0$) we obtain
\begin{equation}
\label{eq:Delta-Holder}
\Delta \le  K^{\gamma/2} M^{\gamma} p^{3\gamma/16+ o(1)}  S^{1/r}  T^{(r-1)/r}, 
\end{equation}
 where
 \begin{align*}
& S =  \sum_{\chi} \left| S_\cJ(\chi)\right|^{2r},\\
& T =  \sum_{\chi}
 \left| S_\cK(\chi)\right|^{(2-\gamma)r/(r-1)} \left| S_\cM(\chi)\right|^{(2-\gamma)r/(r-1)}, 
 \end{align*}
 with $\chi$ running through all multiplicative characters of $\F_p^*$. 

By the orthogonality of characters, we see that $S  = (p-1) U$, where $U$ is the number of solutions to the congruence
$$
j_1 \ldots j_r \equiv j_{r+1} \ldots j_{2r} \mod  p, \qquad
j_1, \ldots, j_{2r} \in \cJ.
$$
Taking into acount that $J^r < p$, we see that this congruence is equivalent to the equation (over $\Z$)
$$
j_1 \ldots j_r  =  j_{r+1} \ldots j_{2r},   \qquad
j_1, \ldots, j_{2r} \in \cJ, 
$$
and thus by the divisor function bound~\eqref{eq: tau}, we have
\begin{equation}
\label{eq:Bound S}
 S \le J^r p^{1+o(1)}.
\end{equation}

We now choose $\gamma$ to satisfy $(2-\gamma)r/(r-1) = 2$ or,  more explicitly,
\begin{equation}
\label{eq:gamma}
\gamma = 2/r.
\end{equation}
Thus, again by   the orthogonality of characters we see that
$$
T =  \sum_{\chi}
 \left| S_\cK(\chi)\right|^2 \left| S_\cM(\chi)\right|^2 = (p-1)V,
$$
where $V$ is the number of solutions to the system of congruences in Lemma~\ref{lem: IntSet-Energy}
and thus we have
\begin{equation}
\label{eq:Bound T}
T \le K^2M^2 p^{o(1)}.
\end{equation}
Substituting the bounds~\eqref{eq:Bound S} and~\eqref{eq:Bound T} in~\eqref{eq:Delta-Holder}
and recalling the choice of $\gamma$ in~\eqref{eq:gamma}, we obtain
\begin{align*}
\Delta & \le  K^{1/r} M^{2/r} p^{3/(8r)+ o(1)}  \(J p^{1/r}\)  \(K^2M^2\)^{(r-1)/r}\\
& =  J K^{2-1/r} M^2  p^{11/(8r)+ o(1)} ,
 \end{align*}
which after the substitution in~\eqref{eq:R Delta} and recalling the choice of $J$,  implies the result.
\end{proof}

\subsection{Multiplicative congruences with arbitrary  intervals}

The following result from~\cite[Lemma~2]{Gar}, relates the number of solutions $J(L,\cH, \cM)$ to the congruence 
$$
(L+h_1)m_1\equiv (L+h_2)m_2\mod  p,\quad h_1,h_2\in \cH,\  m_1,m_2\in\cM.
$$
with  $\cH$ as in~\eqref{eq: init int H},  $\cM$ as in~\eqref{eq: set M} and arbitrary $L\in \Z$ to 
 $J(0, \cH, \cM)=  J(\cH, \cM)$.

\begin{lemma}
\label{lem: sliding to initial int}
For any $L\in \Z$,  we have 
$$
J(L,\cH, \cM) \le 2 J(\cH, \cM) + M^2.
$$
\end{lemma}

For an  integer $s$ we define  $J_s(L, \cH,\cM)$ as the number of solutions to the congruence
$$
m_1x_1^{-s}\equiv m_2x_2^{-s}\mod  p, \quad  m_1, m_2\in  \cM,\  x_1,x_2\in L+\cH.
$$
Note that $J_s(L, \cH,\cM)= J_{-s}(L,\cH,\cM)$ and thus $J_{\pm 1} (L,\cH,\cM)= J(L,\cH,\cM)$, 
where $J(L,\cH,\cM)$ is as in Lemma~\ref{lem: sliding to initial int}. 

We now relate the number of solutions $J_s(L,\cH, \cM)$  to  $J_1(L,\cH,\cM)= J(L,\cH,\cM)$.

\begin{lemma}
\label{lem: s -> 1}
For any  integer $s \ne 0$,  we have 
\begin{align*}
J_s(L,\cH&, \cM)  \\
&\ll \left\{\begin{array}{llll}
H^2M^2/p^{-1}+HMp^{o(1)}, &  \text{if }  H\ge p^{2/3},\\
H^2M^2p^{-1} +HM^{7/4}p^{-1/4+o(1)}+M^2, &  \text{if }   H<p^{2/3}\\
&  \quad  \text{ and } M\ge p^{1/3},\\
HMp^{o(1)}+M^{2}, &  \text{if }   H<p^{2/3}\\
& \quad   \text{ and } M<p^{1/3}.
\end{array}
\right.
\end{align*}
\end{lemma}

\begin{proof} 
For $x,y\in\F_p^*$  we say that $x$ and $y$  equivalent if $xy^{-1}$ is $s$-th power modulo $p$. 
With respect to this  equivalence relation, we can split the set $\cM$ into  $t \le \gcd(s,p-1)$ equivalence  
classes $\cM_1,\ldots,\cM_t$, so that for any two elements
from the same class their ratio is an $s$-th power modulo $p$. Then,
\begin{equation}
\label{eqn: JsM is the sum of JsMj}
J_s(L, \cH,\cM)=\sum_{j=1}^{t} J_s(L, \cH,\cM_j).
\end{equation}

Clearly, for any $j=1, \ldots, t$, there exists an integer $a_j$ and a set $\cN_j \subseteq \F_p^*$  such that
$$
\cM_j =\{a_jn^s: \, n\in\cN_j\},\quad \mand \quad \#\cN_j=\#\cM_j=M_j.
$$
Hence, 
\begin{equation}
\label{eqn: Mj to Nj}
J_s(L, \cH,\cM_j)=J_s(L, \cH,\cN_j),
\end{equation}
where $J_s(\cH,\cN_j)$ is the number of solutions to the congruence
$$
n_1^sx_1^{-s}n_2^{-s}x_2^s\equiv 1\mod  p,\quad n_1,n_2\in\cN_j,\ x_1,x_2\in L+ \cH.
$$
Expressing the number of solution to the congruence via multiplicative characters, see~\cite[Chapter~3]{IwKow}, we obtain 
\begin{align*}
J_s(L, \cH,\cN_j) & = \frac{1}{p-1} \sum_{\chi} \left| \sum_{n\in \cN_j} \chi\(n^s\) \right|^2  
\left| \sum_{x\in L+\cH} \chi\(x^s\)\right|^2 \\
& = \frac{1}{p-1} \sum_{\chi } \left| \sum_{n\in \cN_j} \chi^s\(n\) \right|^2  
\left| \sum_{x\in L+\cH} \chi^s\(x\)\right|^2, 
\end{align*}
where, as before,  $\chi$ runs through all multiplicative characters of $\F_p^*$.

Since any character can be represented as an  $s$-th power of some other character $\chi$ at most 
$\gcd(s,p-1) \le s$ times, we  continue 
\begin{align*}
J_s(L, \cH,\cN_j) & = \frac{s}{p-1} \sum_{\chi} \left| \sum_{n\in \cN_j} \chi\(n\) \right|^2  
\left| \sum_{x\in L+\cH} \chi\(x\)\right|^2 \\
& =  s J(L, \cH,\cN_j). 
\end{align*}
We now use that $M_j \le M$ and recall  Lemmas~\ref{lem: BS-Energy} and~\ref{lem: sliding to initial int}.
We also observe that the extra term $M^2$ in Lemma~\ref{lem: BS-Energy} 
gets absorbed in other terms already present in the bound of Lemma~\ref{lem: BS-Energy}, which together 
with~\eqref{eqn: JsM is the sum of JsMj} and~\eqref{eqn: Mj to Nj} concludes the proof. 
\end{proof}

\subsection{Additive congruences with reciprocals}
We also need the following statement from~\cite[Proposition ~1]{BG}.

\begin{lemma}
\label{lem: BGsym}
Let $\cX$ be an arbitrary interval~\eqref{eq: int X}  of  cardinality $H$. For any fixed positive integer constants $s$ and $\ell$ the number $J_{\ell,s}(\cX)$ of solutions of the congruence
$$
x_1^{-s}+\ldots+x_{\ell}^{-s}\equiv x_{\ell+1}^{-s}+\ldots+x_{2\ell}^{-s} \mod  p,\quad x_1,\ldots,x_{2\ell}\in \cX,
$$
satisfies the bound
$$
J_{\ell,s}(\cX) < \(H^{2\ell^2/(\ell+1)}+\frac{H^{2\ell}}{p}\)H^{o(1)}.
$$
\end{lemma}

\section{Proof of main results} 
\subsection{Proof of Theorem~\ref{thm: HM}}

We choose an integer $r\ge$ such that $1/r < \varepsilon/6$ and 
set 
$$
J =  \fl{p^{1/r}} \mand K =\fl{H/J}.
$$
Then
$$
\cH \cM \supseteq \cJ \cK \cM,
$$
where $\cJ$ and  $\cK$ are as in Lemma~\ref{lem:TripleProd}. 

We now show that $J$, $K$  and $M$ satisfy the conditions of 
 Lemma~\ref{lem:TripleProd}. 

Indeed if 
$$
H<p^{2/3},  \qquad M \ge p^{1/3}, \qquad  HM^{1/4} \ge p^{3/4+\varepsilon} .
$$
then obviously 
$$
K<p^{2/3},  \qquad M \ge p^{1/3}, \qquad  KM^{1/4} \gg p^{3/4+\varepsilon-1/r}
\ge  p^{3/4}
$$
provided that $p$ is large enough. 

Now assume that 
$$
H\ge p^{2/3},  \qquad HM \ge p^{1+\varepsilon} 
$$
and consider two cases. 
If $K \ge p^{2/3}$ then we also have 
$$
KM^{1/4} \gg p^{3/4+\varepsilon-1/r}\ge p^{3/4},
$$
provided that $p$ is large enough. 

However, if $K <  p^{2/3}$ then $H \le p^{2/3 + 1/r}$. Hence, 
$$M \ge p^{1+\varepsilon} H^{-1} \gg  p^{1/3+\varepsilon-1/r}. 
$$
Using $H \ge p^{2/3}$, we obtain $K \ge p^{2/3 -1/r}$ and thus 
$$
KM^{1/4} \gg   p^{2/3 -1/r}  \(p^{1/3+\varepsilon-1/r}\)^{1/4} 
= p^{3/4 + \varepsilon/4 - 5/4r}  \ge p^{3/4}
$$
(again assuming that  $p$ is large enough). 

We note that since $M \le p$ we always have $H \ge p^{1/2}$ and thus taking $r$ 
to be large enough, we obtain $K \ge p^{7/16}$ 

The above verifies the conditions of Lemma~\ref{lem:TripleProd} and thus we get 
\begin{align*}
R(\cJ, \cK, \cM) & = \frac{J^2K^2 M^2} {p-1} \( 1+ O\( K^{-1/r}  p^{3/(8r)+ o(1)}\)\)\\
& =  \frac{J^2K^2 M^2} {p-1} \( 1+ O\(  p^{-\eta}\)\), 
 \end{align*}
with some $\eta > 0$, which depends only on $\varepsilon$. 
Hence, a standard argument, based on  Cauchy's inequality, 
shows that 
$$
\# \(\cJ \cK \cM\) \ge \frac{J^2K^2 M^2}{ R(\cJ, \cK, \cM) } = p + O\(  p^{1-\eta}\), 
$$
which concludes the proof.

\subsection{Proof of Theorem~\ref{thm: KloostFrac}}
Let $\cX = L+\cH$ and let 
$$
S = \sum_{m\in\cM}\left|\sum_{x\in \cX}\e_p(amx^{-s})\right|
$$
be the sum we would like to estimate. 
Using H\"older's inequality and following a well-known strategy, we have that
$$
|S|^{\ell}= \(\sum_{m\in\cM}\left|\sum_{x \in \cX}\e_p(amx^{-s})\right|\)^{\ell}
 \le M^{\ell-1}\sum_{m\in\cM}\left |\sum_{x \in \cX}\e_p(amx^{-s})\right|^{\ell}.
$$

Hence, there exist complex numbers $\vartheta_m$ with $|\vartheta_m|=1$ such that
$$
|S|^{\ell}\le M^{\ell-1}\sum_{m\in\cM}\vartheta_m \sum_{x_1,\ldots,x_{\ell} \in \cX}\e_p(am(x_1^{-s}+\ldots +x_{\ell}^{-s})).
$$

Denoting by $I_{\lambda}$ the number of solutions of the congruence
$$
x_1^{-s}+\ldots+x_{\ell}^{-s}\equiv \lambda\mod  p,\quad x_1,\ldots,x_{\ell} \in \cX,
$$
we get that
$$
|S|^{\ell}\le M^{\ell-1}\sum_{\lambda=0}^{p-1}I_{\lambda}\left|\sum_{m\in\cM}\vartheta_m \e_p(am\lambda)\right|.
$$
Applying  Cauchy's inequality, we obtain 
$$
|S|^{2\ell}\le M^{2\ell-2} \(\sum_{\lambda=0}^{p-1}I_{\lambda}^2\)\sum_{\lambda=0}^{p-1}\left|\sum_{m\in\cM}\vartheta_m \e_p(am\lambda)\right|^2.
$$
Since  
\begin{align*}
\sum_{\lambda=0}^{p-1}&\left|\sum_{m\in\cM}\vartheta_m \e_p(am\lambda)\right|^2\\
&\qquad = \sum_{m_1,m_2\in\cM}\vartheta_{m_1} \overline \vartheta_{m_2} \sum_{\lambda=0}^{p-1} \e_p(a(m_1-m_2)\lambda)
=pM
 \end{align*}  
and 
\begin{align*}
\sum_{\lambda=0}^{p-1}I_{\lambda}^2 = \# \{
x_1^{-s}+\ldots +x_{\ell}^s \equiv x_{\ell+1}^{-s}+\ldots &+x_{2\ell}^{-s} \mod p:\\
& \qquad  x_1,\ldots,x_{2\ell} \in \cX\},
 \end{align*}
by Lemma~\ref{lem: BGsym}, we see that
 \begin{align*}
|S|^{2\ell}& \le  pM^{2\ell-1}\(H^{2\ell^2/(\ell+1)}+\frac{H^{2\ell}}{p}\)H^{o(1)}\\
& =M^{2\ell}H^{2\ell}\(\frac{p}{MH^{2\ell/(\ell+1)}}+\frac{1}{M}\) H^{o(1)},
 \end{align*}
and the desired result follows.

\subsection{Proof of Theorem~\ref{thm: hexagon}}
Let $\cX_i = L_i + \cH$. 
 Expressing $T_6(\lambda)$ in terms of exponential sums, we obtain
$$
T_6(\lambda)= \frac{1}{p}\sum_{a=0}^{p-1} \prod_{i=1}^6 \(\sum_{m\in \cM_i}\sum_{x\in \cX_i} \e_p(amx^{-s})\)\e_p(-a\lambda).
$$
Separating the term  corresponding to $a=0$ and recalling that $0 \notin \cX_i$, we  see that
\begin{align*}
\left |T_6(\lambda) - \frac{H^6M^6}{p}\right|&\le \frac{1}{p}\sum_{a=1}^{p-1} \prod_{i=1}^6 \left| \sum_{m\in \cM_i}\sum_{x\in \cX_i} \e_p(amx^{-s})\right|\\
& \le  W^4\frac{1}{p}
\sum_{a=0}^{p-1} \prod_{i=1}^2 \left| \sum_{m\in \cM_i}\sum_{x\in \cX_i} \e_p(amx^{-s})\right| , 
\end{align*}
where
$$
W=\max_{\gcd(a,p)=1} \max_{3 \le i \le 6} \left| \sum_{m\in \cM_i}\sum_{x\in \cX_i} \e_p(amx^{-s})\right| . 
$$

Note that by Cauchy's inequality 
\begin{align*}
\frac{1}{p}\sum_{a=0}^{p-1} \prod_{i=1}^2 & \left| \sum_{m\in \cM_i}\sum_{x\in \cX_i} \e_p(amx^{-s})\right| \\
& \le\prod_{i=1}^2  \(\frac{1}{p}\sum_{a=0}^{p-1} \left| \sum_{m\in \cM_i}\sum_{x\in \cX_i} \e_p(amx^{-s})\right| ^{2}  \)^{1/2}\\
& = \sqrt{J_s(L_1, \cH,\cM)   J_s(L_2, \cH,\cM)}.
\end{align*}

Hence,
\begin{equation}
\label{eq: T6WJs}
 \left |T_6(\lambda) - \frac{H^6M^6}{p}\right| \ll \Delta , 
\end{equation}
where 
$$
\Delta  =W^4 \sqrt{J_s(L_1, \cH,\cM)   J_s(L_2, \cH,\cM)}.
$$
Hence, it now remains to estimate $\Delta$.

Assume first that $H>p^{2/3}$. Using the  Weil bound  for exponential sums with rational functions (see, for example~\cite{MorMor})
 and completing technique (see~\cite[Section~12.2]{IwKow}), we have 
$$
\sum_{x\in \cX_i} \e_p(amh^{-s}) \ll p^{1/2} \log p, \qquad i=1,2.
$$
Hence, estimating the sums over $\cM_1$ and $\cM_2$ trivially, we obtain
$$
W\ll M p^{1/2}  \log p.
$$
Therefore, substituting this bound in~\eqref{eq: T6WJs} and using Lemma~\ref{lem: s -> 1},  we obtain
\begin{align*}
\Delta&\ll M^4p^2\(\frac{H^2M^2}{p}+HMp^{o(1)}\)\log^4 p  \\
&\ll \frac{H^6M^6}{p} \(\frac{p^2}{H^4}+\frac{p^3}{H^5M}\)p^{o(1)} \\
&  \ll \frac{H^6M^6}{p}p^{-1/3+o(1)},
\end{align*}
and the result follows in the case $H>p^{2/3}$.

Thus, in what follows we assume that $H<p^{2/3}$. In particular, we have that
\begin{equation}
\label{eq: JsJs}
\begin{split} 
\sqrt{J_s(L_1, \cH,\cM)   J_s(L_2, \cH,\cM)}&\\
\ll \frac{H^2M^2}{p}+\frac{HM^{7/4}}{p^{1/4+o(1)}}&+HMp^{o(1)}+M^2. 
\end{split} 
\end{equation}

Applying Theorem~\ref{thm: KloostFrac} to bound $W$, we get that for any  fixed positive integer $\ell$ we have
$$
\Delta\le \frac{H^6M^6}{p}\times R_{\ell} p^{o(1)},
$$
where
$$
R_{\ell} = \frac{p}{H^2M^2}\(\frac{p}{MH^{2\ell/(\ell+1)}}+\frac{1}{M}\)^{2/\ell} \sqrt{J_s(L_1, \cH,\cM)   J_s(L_2, \cH,\cM)}.
$$
Our aim is to prove that for an appropriate choice of $\ell$, we have $R_{\ell}=O(p^{-\delta})$ for   some $\delta>0$ 
which depends on $\ell$, $s$ and $\varepsilon$. 

We choose $\ell=3$, and since $H<p^{2/3}$, we get that
\begin{equation}
\label{eqn: R is after ell =3}
\begin{split} 
R_{3} &  \ll   \frac{p}{H^2M^2} \(\frac{p}{H^{3/2}M}+\frac{1}{M}\)^{2/3} \sqrt{J_s(L_1, \cH,\cM)   J_s(L_2, \cH,\cM)} \\
& \ll  \frac{p}{H^2M^2} \(\frac{p}{H^{3/2}M}\)^{2/3}\sqrt{J_s(L_1, \cH,\cM)   J_s(L_2, \cH,\cM)}.\end{split} 
\end{equation}

We now consider four cases, which depend on what term in~\eqref{eq: JsJs} dominates.

If 
$$
\sqrt{J_s(L_1, \cH,\cM)   J_s(L_2, \cH,\cM)}\ll \frac{H^2M^2}{p},
$$
then from~\eqref{eqn: R is after ell =3} we get that
$$
R_3  \ll \(\frac{p}{H^{3/2}M}\)^{2/3}
$$
and the result follows in view of 
$$
H^{3/2}M> H^{24/17} M^{11/17}>p^{1+\varepsilon}.
$$ 

Next, let 
$$
\sqrt{J_s(L_1, \cH,\cM)   J_s(L_2, \cH,\cM)} \le  HM^{7/4}p^{-1/4+o(1)}.
$$
Then by~\eqref{eqn: R is after ell =3} 
$$
R_3\le \frac{p}{H^2M^2}\(\frac{p}{H^{3/2}M}\)^{2/3} HM^{7/4}p^{-1/4+o(1)} 
\le  \frac{p^{17/12+o(1)} }{H^2M^{11/12}}, 
$$
and the result follows from the condition $H^{24/17}M^{11/17}>p^{1+\varepsilon}$.

Now, let 
$$
\sqrt{J_s(L_1, \cH,\cM)   J_s(L_2, \cH,\cM)}\le HM p^{o(1)} .
$$
Again by~\eqref{eqn: R is after ell =3},  we have that
$$
R_{3} \ll \frac{p}{H^2M^2}\(\frac{p}{H^{3/2}M}\)^{2/3}HM  p^{o(1)} = \frac{p^{5/3+o(1)} }{H^2M^{5/3}}, 
$$
and the result follows from the assumption $H^{6/5}M>p^{1+\varepsilon}$.

Finally, let
$$
\sqrt{J_s(L_1, \cH,\cM)   J_s(L_2, \cH,\cM)}\ll M^2.
$$
In this case~\eqref{eqn: R is after ell =3}  implies that 
$$
R_{3} \ll \frac{p}{H^2M^2}\(\frac{p}{H^{3/2}M}\)^{2/3}M^2 =  \frac{p^{5/3}}{H^3M^{2/3}}
$$
and the result follows from the condition $H^{9/5}M^{2/5}>p^{1+\varepsilon}$ of our theorem.

\section{Comments}

It is easy to see that our argument immediately implies 
the following more general version of Theorem~\ref{thm: KloostFrac}
for weighted sums.  Namely, let  $\balpha = \{\alpha_m\}_{m \in \cM}$ and $\bbeta = \{\beta_x\}_{x\in \cX}$ be complex numbers 
with $|\beta_x|\le 1$. 
Then for an interval $\cH$  as in~\eqref{eq: init int H}  a set $\cM$ as in~\eqref{eq: set M},  for any fixed positive integer $\ell$ we have
\begin{align*}
\left|\sum_{m\in\cM}\sum_{  x\le L+\cH}\alpha_m\beta_x\e_p(amx^{-s})\right| \\
 \le
 \|\balpha \|_{\ell/(\ell-1)}HM^{1/\ell} &\(\frac{p}{MH^{2\ell/(\ell+1)}}+\frac{1}{M}\)^{1/(2\ell)} p^{o(1)},
\end{align*}
where
$$
 \|\balpha \|_{\ell/(\ell-1)}=\(\sum_{m\in\cM}|\alpha_m|^{\ell/(\ell-1)}\)^{(\ell-1)/\ell}.
$$
In turn this bound can lead to   versions of Theorem~\ref{thm: hexagon}
with $x$ running through various subsets of $L+\cH$.  

We also pose an open question of proving that for any $\eps>0$ there exists some 
$\delta > 0$ such that  for $H \ge p^{1/(4e^{1/2}) + \eps}$ and 
$M\ge p^{1-\delta}$ we have an analogue of Theorem~\ref{thm: HM}
with $\#\(\cH \cM\) = p+o(p)$.

\end{document}